\documentclass[11pt,reqno]{amsart}
\usepackage[]{amsmath,amssymb,amsfonts,latexsym,amsthm,enumerate}
\usepackage{amssymb,bbm,mathrsfs,tikz}
\usepackage{enumerate,graphicx,paralist}
\usepackage[font=small]{caption}
\usepackage[noadjust]{cite}

\numberwithin{equation}{section}
\usepackage[colorlinks=true, pdfstartview=FitV, linkcolor=blue,
  citecolor=blue, urlcolor=blue,pagebackref=false]{hyperref}
\usepackage[colorinlistoftodos]{todonotes}
\presetkeys{todonotes}{inline, color=green}{}
\usepackage{comment}

\newtheorem{theorem}{Theorem}[section]
\newtheorem*{theorem*}{Theorem}
\newtheorem{lemma}[theorem]{Lemma}

\newtheorem{proposition}[theorem]{Proposition}

\newtheorem{remark}[theorem]{Remark}
\newtheorem*{remark*}{Remark}

\theoremstyle{definition}{

\newtheorem*{definition*}{Definition}

\newtheorem*{question*}{Question}
\newtheorem*{example*}{Example}
\newtheorem*{examples*}{Examples}
}

\newcommand{\abbr}[1]{{\sc{\lowercase{#1}}}}

\newcommand{\B}{\mathbb B}
\newcommand{\E}{\mathbb E}

\renewcommand{\P}{\mathbb P}
\newcommand{\Q}{\mathbb Q}

\newcommand{\R}{\mathbb R}

\newcommand{\cK}{\mathcal K}
\newcommand{\cA}{{\mathcal A}}

\newcommand{\cE}{{\mathcal E}}
\newcommand{\cF}{{\mathcal F}}

\newcommand{\cM}{{\mathcal M}}

\newcommand{\cS}{{\mathcal S}}

\newcommand{\fs}{\mathfrak s}
\newcommand{\ft}{\mathfrak t}

\newcommand{\one}{\mathbbm{1}}

\renewcommand{\epsilon}{\varepsilon}
\newcommand{\Tr}{\textsc{t}}

\author{Amir Dembo}
\address{Amir Dembo\hfill\break
Mathematics  Department and Statistics Department\\ Stanford University\\ 
Stanford, CA 94305, USA.}
\email{adembo@stanford.edu}

\author{Eyal Lubetzky}
\address{Eyal Lubetzky\hfill\break
Courant Institute 
\\ New York University\\
251 Mercer Street\\ New York, NY 10012, USA.}
\email{eyal@courant.nyu.edu}

\author{Ofer Zeitouni}
\address{Ofer Zeitouni\hfill\break
Department of Mathematics\\
Weizmann Institute of Science\\
Rehovot 76100, Israel\\
and
Courant Institute\\
New York University\\
251 Mercer Street\\ New York, NY 10012, USA.}
\email{ofer.zeitouni@weizmann.ac.il}

\title{Universality for Langevin-like spin glass dynamics}

\begin{document}

\begin{abstract}
We study dynamics for asymmetric spin glass models, proposed by Hertz et al.\ and Sompolinsky et al.\ in the 1980's in the context of neural networks: particles evolve via a modified Langevin dynamics for the Sherrington--Kirkpatrick model with soft spins, whereby the disorder is  i.i.d.\ standard Gaussian rather than symmetric. 
Ben Arous and Guionnet~(1995), followed by Guionnet (1997), proved for Gaussian interactions that  as the number of particles grows, 
 the short-term empirical law of this dynamics converges a.s.\ to a non-random law~$\mu_\star$ of a ``self-consistent single spin dynamics,'' as predicted by physicists. 
Here we obtain universality of this fact:
 For asymmetric disorder given by i.i.d.\ variables of zero mean, unit variance 
and exponential or better tail decay, at every temperature, the empirical law of sample paths of the 
 Langevin-like dynamics in a fixed time interval has the same a.s.\ limit~$\mu_\star$.
\end{abstract}

{\mbox{}\vspace{-0.6cm}
\maketitle
}
\vspace{-0.6cm}

\section{Introduction}\label{secIntro}

Consider the dynamics for asymmetric spin glass models, studied in the context of neural networks e.g.\ by
Hertz et al.~\cite{HGS86} and Cristani and Sompolinsky~\cite{CS87}, given by
\begin{equation}
	\label{eq:main_eq}
	d X_t^{(i)} = d B_t^{(i)} - U_1'(X_t^{(i)}) d t + \frac\beta{\sqrt{N}}\sum_{j=1}^N J_{ij} X_t^{(j)} dt\qquad(i=1,\ldots,N)\,,
\end{equation} 
where $B_t$ is $N$-dimensional Brownian motion, $X_t \in [-\fs,\fs]^N$ for some finite $\fs$, the potential $U_1$ is some smooth function satisfying that $U_1(x)\to\infty$ as $|x|\to\fs$ (e.g.\ a double-well potential at $\pm 1$ with $\fs =2$), the parameter $\beta>0$ is the inverse-temperature and the interactions $J_{ij}$ are quenched (frozen) i.i.d.\ standard Gaussian random variables.

If instead one were to take a symmetric disorder (that is, $J_{ij}=J_{ji}$ i.i.d.\ standard Gaussian for each pair $\{i,j\}$) then the stochastic differential system (\abbr{sds})~\eqref{eq:main_eq} would be precisely Langevin dynamics for the soft-spin Sherrington--Kirkpatrick (\abbr{sk}) model; see, e.g.,~\cite{SZ81,SZ82,MPV87} and~\cite{BG97,BG98,Guionnet97} for studies of the short-term dynamics in that case.

The asymmetric nature of the disorder $J_{ij}$ aids some aspects of the analysis via the extra independence,
yet makes the dynamics non-reversible, whence various useful tools (e.g., the Fluctuation Dissipation Theorem used in~\cite{SZ82} to analyze the symmetric~case) become unavailable. As argued e.g.\ in~\cite{CS87} (see also~\cite{KZ91,DGZ87} on the related Hopfield model~\cite{Hopfield82}), the asymmetric
disorder
seems a better model for the interactions between neurons (cf.\ Remark~\ref{rem:neural} for other flavors of the model in the context of neural networks).

Many of the dynamical quantities of interest, such as spin autocorrelation and response functions, may be read from the 
thermodynamic limit ($N \to \infty$) of the empirical measure $\mu_N$ of
sample-paths of the $N$ particles in a given time interval $[0,T]$; 
that is, 
\begin{equation}\label{eq:def-empirical} 
\mu_N=\frac1N \sum_{i=1}^N \delta_{X_\cdot^{(i)}} \in \cM_1(C([0,T]))\,.\end{equation}
%
Ben~Arous and Guionnet~\cite{BG95}, followed by~\cite{Guionnet97} (cf.~\cite{BG98}) were able to show that, from an i.i.d.\ initial state, $\mu_N$ converges a.s.\ to a law $\mu_\star$ of a \emph{self-consistent single-spin dynamics}, a non-Markovian diffusion involving only one spin,
  as predicted for this system in~\cite{CS87}. The proofs in~\cite{BG95,Guionnet97} (and in follow-up works on variants of this model, e.g., Glauber-like dynamics~\cite{Rieger1989,Grunwald96} and the dynamics where the \abbr{sds} has an extra non-linearity~\cite{FMT19}) relied in an essential way on special properties of Gaussian random variables.
  
In the related \abbr{sk} model, the 
first rigorous proof \cite{GuerraToninelli02,Talagrand03} that the free energy has an a.s.\ limit 
was specific to Gaussian disorder, as was the identification of this limit. Talagrand~\cite{Talagrand02} later proved that the same limit must be obtained 
under interactions of Bernoulli $\pm1$ random variables. This universality property 
was further generalized in~\cite{CarmonaHu06,Chatterjee05} to any i.i.d.\ 
interactions given by a variable $J$ satisfying $\E J=0$ and $\E J^2 =1$.

Our goal in this work is to obtain a similar universality result for the system~\eqref{eq:main_eq},
where a self-consistent a.s.\ limit was till now rigorously verified only in the case of Gaussian interactions.
To be precise, consider the probability measures 
$\P_N^\beta$ of the triplet $({\bf J},B_\cdot,X_\cdot)$ corresponding to the~\abbr{sds}~\eqref{eq:main_eq} with an initial state that is a product~$\nu_0^{\otimes N}$ which places no mass on the boundary, i.e.,
\[ \nu_0 \in \cM_1((-\fs,\fs))\,,
\]
and
the 
$C_2((-\fs,\fs))$ potential function $U_1(x)\to\infty$ as $|x|\to \fs$ fast enough to confine the solution 
of \eqref{eq:main_eq} within 
$(-\fs,\fs)$.
Specifically, suppose (as in~\cite[p.~458]{BG95}), that
\begin{equation}\label{eq:U-assump}
\lim_{|x|\uparrow \fs} \int_{0}^x e^{2U_1(t)} \bigg(\int_0^t e^{-2U_1(v)}dv\bigg) dt = \infty\,,
\end{equation}
which is satisfied for instance by $U_1(x) =  -\log(\fs^2-x^2)$. In this context, the heuristic reasoning 
for the expected universality, as in the case of the free energy in the \abbr{sk} model, is due to the invariance principle, 
whereby one expects the interaction term $N^{-1/2} {\bf J} X_t$ in~\eqref{eq:main_eq} to approximately 
follow a Gaussian law when $N \to \infty$, irrespective of the marginal 
laws of the independent disorder variables $J_{ij}$.
However, even for fully independent (i.e. non-symmetric), Gaussian disorder, the limit $\mu_\star$ of 
$\mu_N$ is characterized only as the global minimum of a certain rate function, corresponding to the variational problem of a large deviation principle (\abbr{ldp}). Consequently, one has
to establish the sought-after universality at the level of large deviations. For $\{J_{ij}\}$ which are
fully i.i.d Gaussian variables and high temperature (i.e. $\beta^2 \fs^2 T < 1 $), such \abbr{LDP} 
was proved in~\cite{BG95} by relying on exact Gaussian calculus for the 
Radon--Nykodim derivative (\abbr{rnd}) w.r.t.\ a reference system with independent particles, 
corresponding to the $\beta=0$ measure (the corresponding a.s. convergence $\mu_N \to \mu_\star$ 
was thereafter extended in \cite{Guionnet97} to all $\beta<\infty$).  Unfortunately, such explicit calculus 
does not exist for any other law of interactions. 
Moreover, any attempt to control the \abbr{rnd} of Gaussian vs.\ non-Gaussian interactions via an argument such as 
Lindeberg's method must be done with utmost care, since it typically 
yields only an $N^{-c}$ additive error term, which is potentially 
multiplied --- and hence outweighed --- by an $e^{c N}$ factor from the \abbr{rnd} (see Remark \ref{rem:explain}).

Our results hold for any random interactions consisting of independent $\{J_{ij}\}$, 
whose laws 
may depend on $i,j,N$, subject only to the following moment and tail assumptions:
\begin{align}\label{eq:J-first-second-mom}
&\E J_{ij} = 0\,,\qquad  \E J_{ij} ^2 = 1\,, \\
&\lim_{\epsilon \to 0} \; 
\sup_{i,j,N} \Big\{ \E \big[ e^{\epsilon |J_{ij}|} \big] \Big\} < \infty \,;
\label{eq:J-assump-unif-mom}
\end{align}
that is, independent $\{J_{ij}\}$ of 
zero mean, unit variance and a uniform exponential, or better, tail  
decay. (In fact, the uniformity over $j$ in \eqref{eq:J-assump-unif-mom} is not
needed and this assumption may be relaxed into the conditions~\eqref{eq:J-assump-mgf}--\eqref{eq:cond-sing-max} stated later; see Remark~\ref{rem:non-uniform-vars}.)

\begin{figure}
    \centering
    \begin{tikzpicture}
    \node (fig1) at (-3.5,0) {
	\includegraphics[width=.45\textwidth]{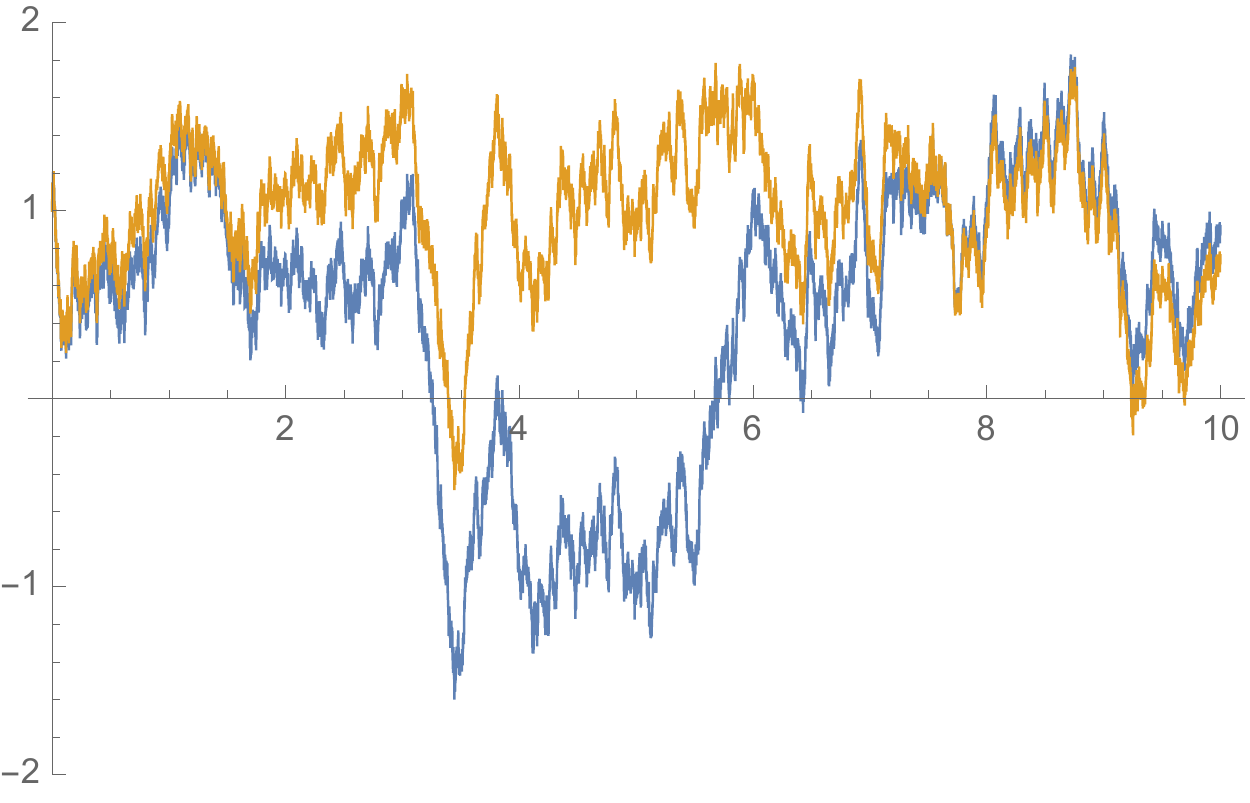}
	};
    \node (fig2) at (3.5,0) {
    \includegraphics[width=.45\textwidth]{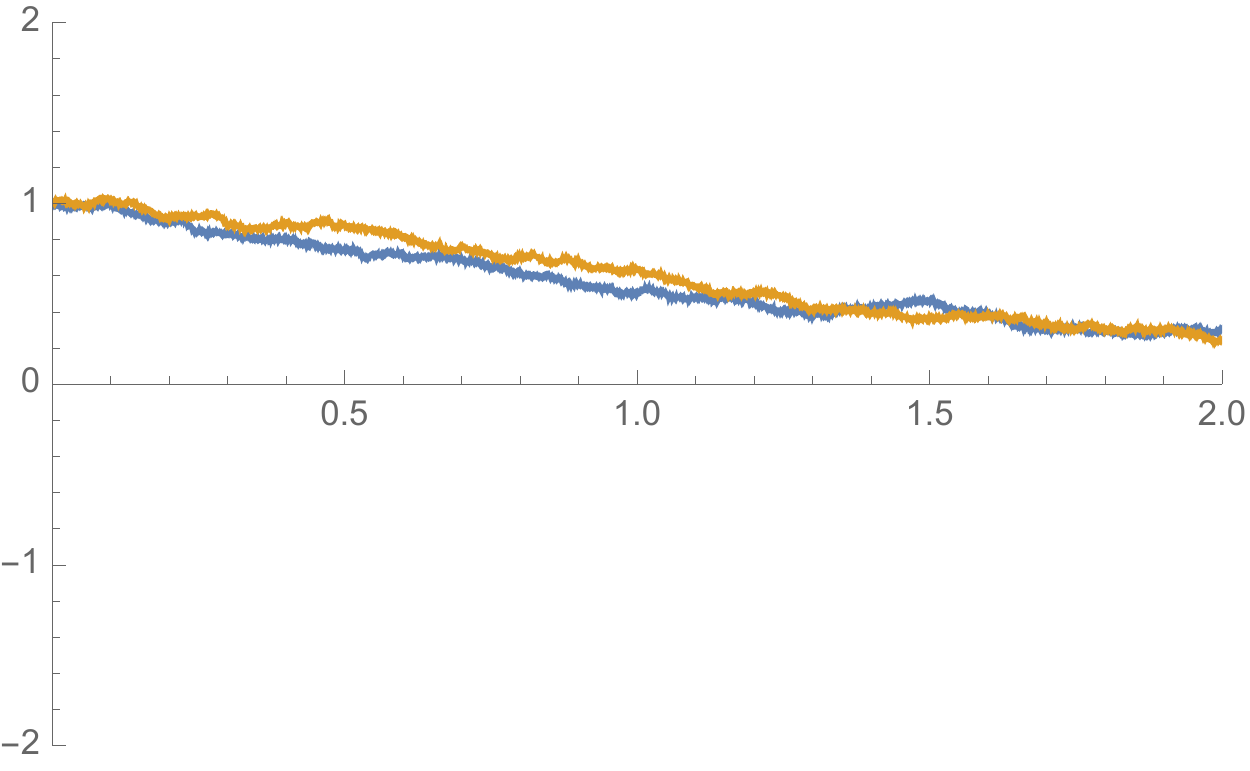}
    };
    \node at (-0.1,0) {$t$};
    \node at (6.9,0) {$t$};
    \end{tikzpicture}
 \vspace{-0.1in}
     \caption{Comparison of the diffusions~\eqref{eq:main_eq} under different
     disorder laws: standard Gaussian $J_{ij}$ (orange) vs.\ Bernoulli $\pm1$ set by $\operatorname{sign}(J_{ij})$ (blue),  with 
 $N=100$ particles, $\beta=1$, $\fs=2$, the double-well potential at $\pm1$ given by 
 $ U_1(x) =  -\log(\fs^2-x^2)-x^2+\tfrac13 x^4$
and a common $N$-dimensional Brownian motion $B_t$ driving the systems. Left: sample path of $X_0^{(I)} X_t^{(I)}$ for a uniform particle $I\in\{1,\ldots,N\}$. Right: average of $100$ samples of $X_0^{(I)} X_t^{(I)}$.}
    \label{fig:sim}
\end{figure}

Let $W_T^\fs$ be the metric space $C([0,T]\to[-\fs,\fs])$ of paths equipped with the distance
\[ d_{2}(x,y) = \bigg( \frac1T \int_0^T |x(t)-y(t)|^2 dt \bigg)^{1/2}\,.\]
Our main result is the a.s.\ convergence of $\mu_N$, in the weak topology corresponding to the metric space $W_T^\fs$, to the self-consistent limit $\mu_\star$ of \cite{BG95,BG98,Guionnet97}.

\begin{theorem}\label{mainthm-1}
Let $\mu_N$ be the empirical measure defined on~\eqref{eq:def-empirical} on sample paths of the Langevin spin glass dynamics~\eqref{eq:main_eq} with independent interactions satisfying~\eqref{eq:J-first-second-mom} and \eqref{eq:J-assump-unif-mom}.
Then, 
for every $\beta>0$, $T<\infty$ and $\fs>0$, we have that a.s.\ 
in the interactions ${\bf J}$ and the diffusion, 
$\mu_N\to \mu_\star$ in $\cM_1(W_T^\fs)$ weakly as $N\to\infty$.
\end{theorem}

As a corollary we obtain, for instance, that the dynamics with interactions that are, e.g., centered Bernoulli($\frac12$) or centered $\mathrm{Exp}(1)$ random variables (see Figure~\ref{fig:sim}) have the same limit as the one derived
in~\cite{BG95,BG98,Guionnet97} for the standard Gaussian case. 

\begin{remark}\label{rem:non-uniform-vars}
Theorem~\ref{mainthm-1} remains valid when replacing the assumption~\eqref{eq:J-assump-unif-mom} by the following, less explicit, yet somewhat more relaxed 
conditions:
\begin{align}\label{eq:J-assump-mgf}
& \lim_{\epsilon \to 0} \sup_{\substack{i \le N\\ \theta \in (0,\epsilon]}} \frac{1}{ \theta^2 N} 
\sum_{j=1}^N \log 
\big( \E [e^{\theta J_{ij}}] \vee \E[e^{-\theta J_{ij}}] \big) < \infty \,, \\
& \frac{1}{N^\gamma} \sum_{i,j=1}^N \E (|J_{ij}|^3) \to 0  \quad \mbox{for some} \quad 
\gamma<\frac{5}{2} \,,
\label{eq:J-assump-moments}\\
& \varlimsup_{N \to \infty} \{ N^{-1/2} \|{\bf J}\|_{2 \to 2} \} < \infty \quad \hbox{almost surely} \,.
\label{eq:cond-sing-max}
\end{align}
\end{remark}

The approach of~\cite{BG95} is to establish a weak \abbr{ldp} for the empirical law of the dynamics, in an approximate system of equations where interactions are frozen 
over a finite number of sub-intervals (see~\eqref{eq:main_eq_frozen}), under the topology derived 
from sup-norm distance between sample paths. One  then boosts it 
via exponential tightness into a full \abbr{ldp}, where
the extra assumption $\beta^2 \fs^2 T < 1$ for exponential tightness  
in \cite{BG95}, is later dispensed of in~\cite{Guionnet97}. The  \abbr{ldp} further 
extends to the original \abbr{sds}, implying in particular the law of large numbers (\abbr{lln}).
Theorem~\ref{mainthm-1} applies beyond Gaussian disorder,  
albeit for the slightly weaker topology derived from $L^2$-distance. 

\begin{remark}\label{rem:explain}
As demonstrated in Figure~\ref{fig:sim}, even when using the same Brownian motion, the sample path for a typical (random) coordinate of the solutions of~\eqref{eq:main_eq} under two different disorder laws are not close to one another: one must average over the disorder matrix $\mathbf{J}$ in order to establish the similarity of the limiting $\mu_N$. Going this route, any attempt to control the \abbr{rnd} between the average of the
measure $\P_N^\beta$ w.r.t.\ our 
non-Gaussian interaction $\mathbf{J}$ and the average of such 
a measure 
with Gaussian interactions~$\widehat{\mathbf{J}}$ requires one to estimate a term of the form $\E_{\mathbf J}[e^F]/\E_{\widehat {\mathbf J}}[e^G]$ conditioned on the sample paths and Brownian motions. The analysis of this \abbr{rnd} becomes particularly delicate since, even upon establishing that $\E_{\mathbf{J},\widehat{\mathbf J}}[e^F - e^G] \leq (1+N^{-c} \, \Xi)^N$, 
we must control the effect of the random variable $\Xi$ in order to deduce that the overall ratio is 
$\exp(o(N))$.
\end{remark}

\begin{remark}\label{rem:neural}
Various extensions of the model studied here appeared in the context of disordered neural networks. For instance, in~\cite{CT18a,CT18b} (also see~\cite{CT13}) the model allows time delays in the interaction between the particles, a time-dependent self-interaction, and any bounded Lipschitz-continuous pairwise interaction (which in the setup of~\cite{BG95,BG98,Guionnet97} was a bi-linear map). In studies of networks of Hopefield neurons, e.g.~\cite{FMT19} and the references therein, the evolution of $X_t^{(i)}$ has interaction terms $J_{ij}$ as pre-factors of a nonlinear uniformly bounded function of the $X_t^{(j)}$'s, in lieu of a confining potential $U$. Both of these lines of extensions were studied under the assumption that the interactions variables $J_{ij}$ are Gaussian, and it would be interesting to analyze them under general interactions. It is plausible that our methods here would be useful to that end.
\end{remark}

In Section \ref{sec:main-prf} we describe the \abbr{sds} of piecewise 
frozen interactions and establish that Theorem~\ref{mainthm-1} is a direct 
consequence of Propositions~\ref{thm-disc} and~\ref{thm-coupling}. 
Thereafter, in Section~\ref{sec:frozen} we establish 
Proposition~\ref{thm-disc}, namely the relevant \abbr{lln} for 
the approximating \abbr{sds}, whereas in Section~\ref{sec:couple} 
we prove Proposition~\ref{thm-coupling}, which couples the 
approximating \abbr{sds} to 
the (original) dynamics~\eqref{eq:main_eq} of interest.

\section{Proof of Theorem \ref{mainthm-1}:
piecewise frozen interactions}\label{sec:main-prf}

We start by showing that the conditions in Remark~\ref{rem:non-uniform-vars} indeed relax~\eqref{eq:J-first-second-mom}--\eqref{eq:J-assump-unif-mom}.
\begin{lemma}\label{prop:disorder}
Conditions~\eqref{eq:J-first-second-mom} and~\eqref{eq:J-assump-unif-mom} 
imply the conditions \eqref{eq:J-assump-mgf}--\eqref{eq:cond-sing-max}.
\end{lemma}
\begin{proof} Taking the expectation of 
\[
e^{\theta J_{ij}} - \theta J_{ij} \le 1 + \frac{\theta^2}{\epsilon^2} e^{\epsilon |J_{ij}|} \,, 
\quad \forall |\theta| \le \epsilon 
\]
w.r.t.\ the zero-mean law of 
$J_{ij}$, followed by 
the logarithm of both sides, as $\log (1+y) \le y$ on $\R_+$ it follows 
that 
\[
\log \E [ e^{\theta J_{ij}} ] \le \frac{\theta^2}{\epsilon^2} \E[e^{\epsilon |J_{ij}|}] 
\,, \quad \forall |\theta| \le \epsilon, i,j,N \,.
\]
Thereby, \eqref{eq:J-assump-mgf} follows from \eqref{eq:J-assump-unif-mom}. 
Similarly, with $|J|^3 \le \frac{6}{\epsilon^3} e^{\epsilon |J|}$, upon 
taking the expectation of both sides w.r.t.\ the law of $J_{ij}$, we get \eqref{eq:J-assump-moments}
(for $\frac{5}{2} > \gamma > 2$). As for \eqref{eq:cond-sing-max}, 
let ${\bf A} := \frac{\beta}{\sqrt{N}} {\bf J}$ 
denote the scaled disorder matrix and 
${\bf Z}_1$, ${\bf Z}_2$ be two $N$-dimensional 
symmetric matrices, which are independent of ${\bf A}$ and 
of each other, with independent entries above and on their
main diagonal satisfying both \eqref{eq:J-first-second-mom} and 
\eqref{eq:J-assump-unif-mom}. Then, for non-random $\gamma \in \R$ 
consider the $2N$-dimensional symmetric matrices
\[
{\bf W}_\gamma := \begin{pmatrix} 
\frac{\gamma}{\sqrt{N}} {\bf Z}_1 & {\bf A}\\
{\bf A}^\Tr& \frac{\gamma}{\sqrt{N}} {\bf Z}_2
\end{pmatrix}\,,
\]
noting that 
\[
\| {\bf A}\|_{2 \to 2} = \lambda_{\max}({\bf W}_0) \le 
\lambda_{\max}({\bf W}_\beta) - \beta 
\lambda_{\min} (N^{-1/2} {\bf Z}_1) - \beta \lambda_{\min}(N^{-1/2} {\bf Z}_2) \,.
\]
But
${\bf W}_\beta$ is a $\sqrt{2}\beta$ multiple of an $2N$-dimensional 
Wigner matrix while $N^{-1/2} {\bf Z}_i$ for $i=1,2$, are 
a pair of $N$-dimensional Wigner matrices.
The F\"{u}redi-Koml\'{o}s~\cite{FK81} argument applies to each of these three
matrices, yielding that 
$\varlimsup_{N\to\infty} \{\lambda_{\max}({\bf W}_\beta) \} \leq 2 \sqrt{2} \beta$ and 
$\varliminf_{N\to\infty} \{ \lambda_{\min} (N^{-1/2} {\bf Z}_i) \} \geq - 2$ for $i=1,2$. 
This completes the proof.\footnote{The result of~\cite{FK81} is for matrices of bounded entries and convergence in probability; the proof remains valid under the condition of uniform boundedness of exponential moments of the entries of~$\sqrt{N}A$: see
the detailed exposition in
\cite[Sec.~2.1.6]{AGZ} for the case of i.i.d.\ entries. The extension to
non-identically distributed entries and a.s.\ convergence, respectively, is immediate (see~\cite[Ex.~2.1.27 and~2.1.29]{AGZ}).}
\end{proof}

A key ingredient in our proof is the analysis of the approximate dynamics of~\cite[\S3]{BG95}, now
for a general disorder $\{J_{ij}\}$. Specifically, fixing an integer $\kappa$, let 
\[\mathfrak t_k = k T/\kappa \qquad\mbox{ for $k=0,\ldots,\kappa$}\,,\]
partitioning the interval 
$[0,T)$ into $\kappa$ disjoint sub-intervals $ [\ft_{k-1},\ft_k)$. We denote by 
$\widetilde{\P}_{N,\kappa}^\beta$ the probability measure of the triplet $({\bf J},B_\cdot,\widetilde{X}_\cdot)$ corresponding to the diffusion $\widetilde{X}_t$ starting from $\widetilde{X}_0=X_0$ and given by
\begin{equation}\label{eq:main_eq_frozen}
 d \widetilde X_t  = d B_t - \nabla U(\widetilde X_t) dt + \frac{\beta}{\sqrt{N}} {\bf J} \widetilde X_{\ft_{k-1}} dt\qquad (t\in[\ft_{k-1},\ft_{k}],\, 1\leq k\leq\kappa)\,,
\end{equation}
i.e., the interaction term between the particles is frozen along each sub-interval $[\mathfrak t_{k-1},\mathfrak t_k)$.
(See Figure~\ref{fig:disc} for a simulation of the approximate dynamics.)

\begin{figure}
    \centering
    \begin{tikzpicture}
    \node (fig1) at (-3.5,0) {
	\includegraphics[width=.65\textwidth]{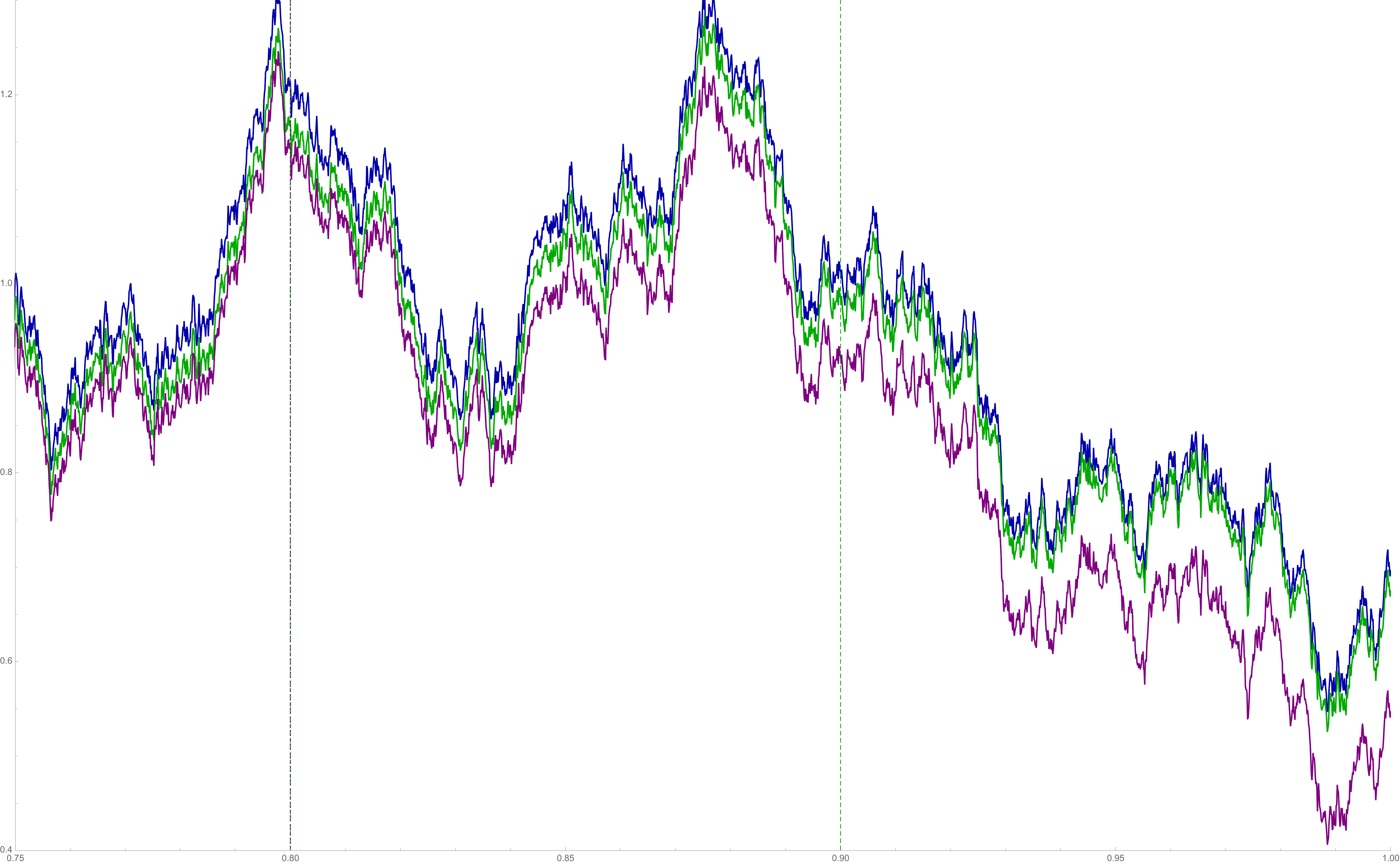}
	};
    \end{tikzpicture}
 \vspace{-0.1in}
     \caption{Solution of the diffusion~\eqref{eq:main_eq} with Bernoulli $\pm1$ disorder variables,
  $N=100$ and $\beta=1$ (blue) vs.\ the approximating system~\eqref{eq:main_eq_frozen} with subinterval length $0.1$ (green) and subinterval length $0.2$ (purple).}
    \label{fig:disc}
\end{figure}

The fact that both \eqref{eq:main_eq_frozen} and the original 
diffusion~\eqref{eq:main_eq} have unique weak solutions, 
follows from~\cite[Proposition 2.1]{BG95}, which established this fact for every $(J_{ij})$. Furthermore, this solution is in fact strong (see, e.g.,~\cite[exercises~(2.10)($1^\circ$) and~(2.15)($2^\circ$) in p.~383 and p.~386]{RY99}).

Next, per $a_2$ finite, denote by $\P_{N}^{\beta,a_2}$ 
the measure $\P_N^\beta$ restricted to the event 
\begin{equation}\label{eq:Astar}
\cA_{a_2} :=  \{ \|{\bf A}\|_{2\to 2}\leq a_2 
\} \,.
\end{equation}
We further use $\Pi_N^\beta$ for the averaged over ${\bf A}$ law of the
empirical measure $\mu_N$, with $\Pi_N^{\beta,a_2}$ similarly standing for 
the sub-probability measure in which the expectation over ${\bf A}$ 
is restricted to the event $\cA_{a_2}$.
In analogy with~\eqref{eq:def-empirical}, let $\widetilde\mu_{N,\kappa}$ be the empirical measure of the solution to~\eqref{eq:main_eq_frozen}, with $\widetilde\Pi_{N,\kappa}^{\beta,a_2}$ denoting its law 
integrated over the disorder restricted to $\cA_{a_2}$.

Recall that $W_T^\fs$ is the metric space $C([0,T]\to[-\fs,\fs])$ equipped with 
the distance
\[ d_{2}(x,y) = \bigg( \frac1T \int_0^T |x(t)-y(t)|^2 dt \bigg)^{1/2}\,.\]
We further equip the space $\cM_1(W_T^\fs)$ 
with the corresponding Wasserstein metric
\[ d_{W_2}(\phi,\psi) := \inf_{\substack{\xi=(\xi_1,\xi_2)\\ \xi_1=\phi,\xi_2=\psi}} \bigg\{ \int d_{2}(x,y)^2 d \xi(x,y)\bigg\}^{1/2}\,,
\]
denoting hereafter by $\B(\mu_\star,\delta)$ the 
ball of radius $\delta$ around $\mu_\star$ in that metric.
\begin{proposition}\label{thm-disc}
Suppose \eqref{eq:J-first-second-mom},
\eqref{eq:J-assump-mgf} and \eqref{eq:J-assump-moments} hold.
Then, for every $T,a_2<\infty$ and $\delta>0$ there exists some $\kappa_0<\infty$ such that for every~$\kappa\geq \kappa_0$,
\[ \sum_{N \geq 1}^\infty \widetilde\Pi_{N,\kappa}^{\beta,a_2} (\B (\mu_\star,\delta)^c) < \infty \,.
\]
\end{proposition}

Next, let $\Q^{\beta,a_2}_{N}$ denote the joint law of $\bf J$, 
$\widetilde{X}_t$ and $X_t$, restricted to $\cA_{a_2}$,
where we use the same $N$-dimensional Brownian motion $B_t$ for both processes.

\begin{proposition}\label{thm-coupling}
Suppose \eqref{eq:J-first-second-mom}, \eqref{eq:J-assump-mgf} and 
\eqref{eq:J-assump-moments} hold. Then, for every $T, a_2<\infty$ 
and $\delta>0$, there exists some $\kappa_0<\infty$ such that for every $\kappa\geq \kappa_0$,
\[ \sum_{N\geq 1}^\infty \Q^{\beta,a_2}_N\bigg( \frac1{NT}\int_0^T\|X_t- \widetilde X_t\|^2_{2}dt> \delta \bigg) < \infty\,.\]
\end{proposition}
\noindent
Coupling each coordinate of $X_t$ with the corresponding one of 
$\widetilde X_t$, one has that
\[ d_{W_2}(\mu_N,\widetilde\mu_{N,\kappa})^2 \leq \frac1N \sum_{i=1}^N d_{2}(X^{(i)},\widetilde{X}^{(i)})^2 = \frac1{N T}  \int_0^T\|X_t- \widetilde X_t\|^2_{2}
dt \,.
\]
Thus, combining Proposition~\ref{thm-disc}, Proposition~\ref{thm-coupling} 
and the triangle inequality for $d_{W_2}(\cdot,\cdot)$ we have that
for any $T$ finite, $a_2$ finite and $\delta>0$
\[
\sum_{N\geq 1}^\infty \P^\beta_{N}
\Big( d_{W_2} (\mu_N,\mu_\star) > 2 \sqrt{\delta}, \cA_{a_2} \Big) < \infty\,,
\]
which by Borel--Cantelli I, implies that for any $\delta>0$ and $a_2$ finite,
\[
\P^\beta \Big[ \varlimsup_{N \to \infty} \{ d_{W_2} (\mu_N,\mu_\star) \} > 2 \sqrt{\delta}\,,
\;\; 
\varlimsup_{N \to \infty} \|{\bf A}\|_{2 \to 2} < a_2 \;\; 
\Big] = 0 \,.
\]
In view of \eqref{eq:cond-sing-max}, the proof of 
Theorem~\ref{mainthm-1} is thus complete.

\section{Proof of Proposition~\ref{thm-disc}}\label{sec:frozen} 
Our proof relies on the following 
application of the multivariate Lindeberg's method of \cite[Theorem 1.1]{Chatterjee06}.
\begin{lemma}\label{lem:lindeberg}
Suppose the random vector ${\bf J} \in \R^N$ has
independent entries $\{J_j\}$ such that $\E J_j =0$ and
$\E J_j^2 =1$.
Then, for every $N,\kappa\geq 1$, non-random ${\bf X} = (x_{kj}) \in \R^{\kappa\times N }$, ${\bf b} \in \R^\kappa$, and the quadratic function 
\begin{equation}\label{eq:quad}
 h(z) = \frac{1}{2} \left\| {\bf X} z  - {\bf b} \right\|_2^2  \qquad(z\in\R^N)\,,
\end{equation}
setting $c_0 = 
\frac12 e^{-\sqrt{3}/2}(3^{\frac14}+3^{-\frac14})$ one has that 
\begin{equation}\label{eq:bd-lind}
| \E e^{-h({\bf J})} - \E e^{-h({\bf \widehat J})}  | \le c_0
\sum_{j=1}^N ({\bf X^\Tr X})_{jj}^{3/2}
(\E |J_j|^3 + \E |\widehat J|^3 ) 
\end{equation}
where  $\widehat {\bf J}=(\widehat J_j) \in \R^N$ has i.i.d.\ standard Gaussian entries. 
In addition,
\begin{equation}\label{eq:gauss-lbd}
\E \big[ \exp(h({\bf 0})-h({\bf \widehat J})) \big] \ge \det ({\bf I} + {\bf X X^{\Tr}})^{-1/2} \,.
\end{equation}
\end{lemma}
\begin{proof} Having mutually independent entries of $\bf J$ whose
first and second moments match those of $\bf \widehat J$, eliminates 
the first two terms of the bound on the \abbr{lhs} of \eqref{eq:bd-lind}
that we get by applying \cite[Theorem 1.1]{Chatterjee06} for the smooth function 
$f(z)=e^{-h(z)}$. Denoting the
first three partial derivatives of a function $f$ w.r.t.\ $z_j$ by $f_j$, $f_{jj}$ and $f_{jjj}$,
the proof of \cite[Theorem 1.1]{Chatterjee06} 
provides a sharper bound than stated in its last term, namely
\begin{equation*}
| \E f({\bf J}) - \E f ({\bf \widehat J})  |\leq \frac{1}{6} \sum_{j=1}^N \|f_{jjj}\|_\infty \, (\E |J_j|^3 + \E |\widehat J|^3 ) \,.
\end{equation*}
For $h(z)$ of \eqref{eq:quad}, we have $\nabla h = {\bf X}^{\Tr} ({\bf X} z-{\bf b})$, 
so $h_{jj} = ({\bf X^\Tr X})_{jj}$ is constant, with $h_{jjj}=0$
and $|h_j| \le \sqrt{2 h h_{jj}}$ by Cauchy--Schwarz. Substituting $r=\sqrt{2h}$ we thus 
have that 
\[
|(e^{-h})_{jjj}| = |h_{jjj} - 3 h_j h_{jj} + h_j^3| e^{-h} 
\le h_{jj}^{3/2} \sup_{r \ge 0} \left\{e^{-\frac{1}{2} r^2} (3 r + r^3) \right\} 
= 6c_0 h_{jj}^{3/2} \,,
\]
from which the \abbr{rhs} of \eqref{eq:bd-lind} follows. To get
\eqref{eq:gauss-lbd} note that the multivariate Gaussian ${\bf g} := \bf X \widehat J$ has 
zero mean and covariance ${\bf X X^\Tr}$. Consequently,
\[
\E \big[ e^{h({\bf 0})-h({\bf \widehat J})} \big] = \E
\big[ e^{-\frac{1}{2} \|{\bf g}\|_2^2 + \langle {\bf g} , {\bf b} \rangle} \big]
\ge \E \big[ e^{-\frac{1}{2} \|{\bf g}\|_2^2} \big] = \det \big(
{\bf I} + \E [{\bf g} {\bf g}^{\Tr}] \big)^{-1/2} \,,
\]
as claimed in \eqref{eq:gauss-lbd}.
\end{proof}

Let $\widehat\P_{N,\kappa}^\beta$ and $\widehat\Pi_{N,\kappa}^\beta$ be the counterparts of $\widetilde\P_{N,\kappa}^\beta$ and $\widetilde\Pi_{N,\kappa}^\beta$ when the disorder~${\bf J}$ is replaced by $\widehat{\bf J}$ whose entries are i.i.d.\ standard Gaussian random variables. Fixing the random variables
\begin{equation}\label{dfn:si}
M_\kappa^{(i)}:=\frac{1}{2} \|{\bf b}^{(i)}\|^2_2, \qquad
b_k^{(i)} := 
\frac{\widetilde X^{(i)}_{\ft_{k}} - \widetilde X^{(i)}_{\ft_{k-1}}-\int_{\ft_{k-1}}^{\ft_k} U_1'(\widetilde X_s^{(i)}) ds}{\sqrt{\ft_k - \ft_{k-1}}} \,,\quad \; k=1,\ldots,\kappa,
\end{equation}
we control $\widetilde \Pi^{\beta}_{N,\kappa}$ 
for some $\delta^{(i)}_N \to 0$ 
in terms of its 
counterpart $\widehat \Pi^{\beta}_{N,\kappa}$ and 
\begin{equation}\label{dfn:Phi-N}
\Phi_{N,\kappa} :=  \frac1N \sum_{i=1}^N \log\big(1+ \delta^{(i)}_N e^{M_\kappa^{(i)}} \big),
\end{equation}
where 
\begin{equation}\label{dfn:delta-n}
\delta^{(i)}_N = c_1 N^{-3/2} \sum_{j=1}^N \E |J_{ij}|^3 \,.
\end{equation}
\begin{lemma}\label{lem:compare}
Assume the independent $\{J_{ij}\}$ satisfy \eqref{eq:J-first-second-mom}.
Then, for any $T,\beta,\kappa$ there exist $N_0$ and $c_1$ finite, such that for every $N\geq N_0$,
\[ 
\frac{{\rm d} \,\widetilde \Pi_{N,\kappa}^{\beta}}{{\rm d}\, \widehat \Pi_{N,\kappa}^{\beta}}
\leq e^{N \Phi_{N,\kappa}}\,.
\]
\end{lemma}
\begin{proof}
For independently and uniformly chosen $I_N \in \{1,\ldots,N\}$, let
\begin{equation}\label{dfn:Gamma}
\Gamma_{N,\kappa}^\beta(\widetilde{\mu}_{N,\kappa}) := \E_{I_N} \log \E_{\bf J} \bigg[ 
\exp \Big(\langle {\bf b}^{(I_N)}, {\bf g}^{(I_N)} \rangle - \frac{1}{2} \|{\bf g}^{(I_N)}\|_2^2 \Big)
\bigg] \,,
\end{equation}
where $J_{ij}$ are independent and 
the coordinates of each ${\bf g}^{(i)} \in \R^{\kappa}$ ($i=1,\ldots,N$) are
\begin{equation}\label{dfn:gt}
g^{(i)}_k := \sum_{j=1}^N x_{kj} J_{ij} \,, \quad
x_{kj} := \frac{\beta \sqrt{T}}{\sqrt{N \kappa}}
\widetilde{X}^{(j)}_{\ft_{k-1}} \,,
\quad k=1,\ldots,\kappa\,.
\end{equation}
 We further define 
$\widehat \Gamma_{N,\kappa}^\beta(\widetilde{\mu}_{N,\kappa})$
as in \eqref{dfn:Gamma}--\eqref{dfn:gt}, except for using $\{\widehat {\bf g}^{(i)}\}$ 
and the i.i.d.\ standard normal variables $\{ \widehat J_{ij} \}$ instead of 
$\{{\bf g}^{(i)}\}$ and $\{J_{ij}\}$, respectively. Note that by Girsanov's theorem  we have
the Radon--Nykodim derivative
\begin{equation}\label{eq:rdn-beta}
\frac{{\rm d} \,\widetilde \P_{N,\kappa}^\beta}{{\rm d}\, \widetilde\P_{N,\kappa}^0} 
= \exp\Big( \sum_{i=1}^N \big[ \langle {\bf b}^{(i)}, {\bf g}^{(i)} \rangle 
- \frac{1}{2} \|{\bf g}^{(i)}\|_2^2 \big] \Big) \,,
\end{equation}
where Novikov's condition holds since 
$$
\widetilde\E_{N,\kappa}^0\bigg( \exp\Big(\frac{1}{2} \sum_{i=1}^N\|{\bf g}^{(i)}\|_2^2\Big)\;\Big|\;{\bf J}\bigg)<\infty \,,
$$
due to the uniform bound on $\{x_{kj}\}$ of \eqref{dfn:gt}
 (as $\| \widetilde X_t \|_\infty \le \fs$).
Under $\widetilde \P^0_{N,\kappa} = \P^0_N$ we have that ${\bf J}$ is independent of 
$(\widetilde X_{\cdot})$,
thereby yielding that
\begin{equation}\label{eq:rnd}
\E_{\bf J} \bigg[ \frac{{\rm d} \,\widetilde \P_{N,\kappa}^\beta}{{\rm d}\,\P_{N}^0} \bigg]
 = \exp(N \Gamma_{N,\kappa}^\beta(\widetilde{\mu}_{N,\kappa}) )\,,
 \end{equation}
where we also crucially used the independence of the rows of ${\bf J}$
to arrive at the specific form on \abbr{rhs}. Being a function of only $\widetilde \mu_{N,\kappa}$, the \abbr{rhs} of 
\eqref{eq:rnd} 
coincides with the Radon--Nykodim derivative restricted to these empirical measures, namely
\[
\frac{{\rm d} \,\widetilde \Pi_{N,\kappa}^{\beta}}{{\rm d}\,\Pi_{N}^{0}}
= \exp(N \Gamma_{N,\kappa}^\beta(\widetilde{\mu}_{N,\kappa}) ) \,.
\]
The same argument applies for 
the Radon--Nykodim derivative of $\widehat \Pi^{\beta}_{N,\kappa}$ with respect 
to~$\Pi^0_{N}$. 
To complete the proof it thus 
suffices to show that
\begin{align}\label{eq:Gamma-Gamma-hat-compare}
    \Gamma_{N,\kappa}^\beta(\widetilde{\mu}_{N,\kappa}) -
\widehat \Gamma_{N,\kappa}^\beta(\widetilde{\mu}_{N,\kappa}) \le \Phi_{N,\kappa} \,.
\end{align}
Unraveling \eqref{dfn:si}--\eqref{dfn:gt}
this follows upon showing that for each $1 \le i \le N$,
\begin{equation}\label{eq:lind-i}
\E \big[ e^{-h^{(i)}({\bf J}^{(i)})} \big] - \E \big[ e^{-h^{(i)}({\bf \widehat J})} \big]
\le \delta^{(i)}_N e^{h^{(i)}({\bf 0})} \E \big[ e^{-h^{(i)}({\bf \widehat J})} \big] \,,
\end{equation}
where 
${\bf \widehat J}$ is a standard multivariate Gaussian, 
${\bf J}^{(i)} = (J_{i1},\ldots,J_{iN}) \in \R^N$
and $h^{(i)}(\cdot)$ of \eqref{eq:quad} with ${\bf b}^{(i)}$ of \eqref{dfn:si} and 
$\{x_{kj}\}$ of \eqref{dfn:gt}. Since 
$\widetilde X_t^{(j)} \in [-\fs,\fs]$, we have that 
\[
|x_{kj}|^2 \le \frac{(\beta \fs)^2 T}{\kappa N} \quad
\Longrightarrow \quad ({\bf X^{\Tr} X})_{jj} \le \frac{(\beta \fs)^2 T}{N}, \;\;\;
({\bf X X^{\Tr}})_{kk'} \le \frac{(\beta \fs)^2 T}{\kappa} \,.
\]
Thus, from Lemma~\ref{lem:lindeberg} the \abbr{rhs} of \eqref{eq:gauss-lbd} 
is bounded below in our case by $1/c_2$ for some $c_2= 
c_2(\beta \fs \sqrt{T},\kappa)$ finite, while 
for some $c_3 = 
c_3(\beta \fs \sqrt{T})$ finite,
the \abbr{lhs} of \eqref{eq:lind-i} is
at most
$c_3 N^{-3/2} \sum_{j \le N} (\E |J_{ij}|^3 + \E |\widehat{J}|^3 )$.
With $\E |J_{ij}|^3 \ge 1$ and $\E|\widehat{J}|^3 = \sqrt{8/\pi}$,
taking $c_1 = c_2 c_3 (1+\sqrt{8/\pi})$ in 
\eqref{dfn:delta-n} guarantees that \eqref{eq:lind-i} 
would hold and thereby completes the proof of the lemma.
\end{proof}

The following elementary lemma is needed for proving 
Lemma \ref{lem:Phi-summable} (namely, to show that
$\Phi_{N,\kappa} \to 0$ a.s.\ when $N \to \infty$).
\begin{lemma}\label{lem:mgf-quadratic}
Suppose vectors ${\bf J}=(J_1,\ldots,J_N) \in \R^N$ are composed
of independent coordinates $\{J_i\}$ such that for some 
$\epsilon >0$, $v<\infty$, and all $N \ge N_0$, 
\begin{equation}\label{eq:expon-tail} 
\sup_{\theta \in (0,\epsilon]} 
\Big\{ \frac{1}{\theta^2 N} \sum_{j=1}^N 
\log \big( \E [e^{\theta J_j}] \vee \E[e^{-\theta J_j}] \big) \Big\} \le v \,.
\end{equation}
For any $a<\infty$, if $ \alpha < \frac{1}{4 v} \wedge \frac{\epsilon}{4 a}$ and 
$N \ge N_1 := N_0 \vee \frac{2}{\epsilon a}$, then
\[
\sup_{\{ {\bf u} \in \R^N \,:\;
 \|{\bf u} \|_\infty\le N^{-1/2} \} } \, \E \Big[
\exp \big( \alpha \langle {\bf u}, {\bf J} \rangle^2 \big) \,
{\one}_{ \{ \|{\bf J}\|_1 \le a N \} } \Big] \le f_\star(\alpha v) < \infty \,. 
\]
\end{lemma}
\begin{proof} Fixing $\alpha>0$,
associate with each non-random 
${\bf u} \in \R^N$ such that $\|{\bf u}\|_\infty \le N^{-1/2}$
the variable 
$Y_{\bf u} := \sqrt{2 \alpha} \langle {\bf u}, {\bf J} \rangle$, noting 
that for $\alpha \le \epsilon/(4a)$ and any such ${\bf u}$
\[
\{ \| {\bf J} \| \le a N \} \quad \Longrightarrow \quad
|Y_{\bf u}| \le a \sqrt{2 \alpha N} \le \frac{\epsilon}{2} \sqrt{N/(2\alpha)} := r_N \,.
\]
Taking $N \ge N_0$ yields in view of \eqref{eq:expon-tail}
(at $\theta = \lambda \sqrt{2\alpha/N}$), 
that
\begin{equation}\label{eq:Yu-bd}
\E \Big[ e^{\lambda Y_{\bf u} } \Big] 
\le e^{2 \alpha v \lambda^2} \,, \qquad \forall 
|\lambda| \le 2 r_N \,.
\end{equation}
Recall the elementary bound, valid for all $r \ge 1$
\begin{equation}\label{eq:elem-G}
e^{y^2/2} \one_{[-r,r]}(y) \le 2 \int_{-2r}^{2r} e^{\lambda y} e^{-\lambda^2/2} \frac{d \lambda}{\sqrt{2\pi}} \,.
\end{equation}
Further, since $r_N \ge \sqrt{\epsilon a N/2} \ge 1$ for all $N \ge N_1$, upon 
combining the bounds 
\eqref{eq:Yu-bd}, \eqref{eq:elem-G} with Fubini's theorem, we find that
for any such $\alpha$, $Y_{\bf u}$ and for all $N \ge N_1$, 
\[
\E \Big[ e^{Y_{\bf u}^2/2} \one_{\{ \|{\bf J}\|_1 \le a N \}} \Big] \le 2 
\int_{-\infty}^\infty e^{2 \alpha v \lambda^2}e^{-\lambda^2/2} 
\frac{d \lambda}{\sqrt{2\pi}} := f_\star(\alpha v) < \infty \,,
\]
as claimed.
\end{proof}

Equipped with Lemma \ref{lem:mgf-quadratic} we
proceed to verify that a.s.\ $\Phi_{N,\kappa} \to 0$ when $N \to \infty$.
\begin{lemma}\label{lem:Phi-summable}
Suppose the independent variables $\{J_{ij}\}$ satisfy~\eqref{eq:J-assump-mgf} and~\eqref{eq:J-assump-moments}. 
Then, for any $T,\beta,a_2,\kappa$ and all $\eta>0$,
\begin{equation}\label{eq:PhiN-summable}
\sum_{N=1}^\infty  \widetilde \P^{\beta,a_2}_{N,\kappa} 
(\Phi_{N,\kappa} > 2 \eta) < \infty \,.
\end{equation}
\end{lemma}
\begin{proof} Set 
${\widehat M}_\kappa^{(i)}:=\frac{1}{2} \|{\bf \widehat b}^{(i)}\|^2_2$ for
$\widehat b_k^{(i)} := b_k^{(i)} - g_k^{(i)}$. Then, for 
any $q \in (0,1]$ and $r_N \ge 0$,
\begin{equation}\label{eq:M-to-whM}
M_\kappa^{(i)} \le  (1+q) {\widehat M}_\kappa^{(i)} +
q^{-1} \| {\bf g}^{(i)} \|_2^2
\le (1+q) {\widehat M}_\kappa^{(i)} + r_N 
+ q^{-1} \| {\bf g}^{(i)} \|_2^2 
\one_{\{\| {\bf g}^{(i)} \|_2^2 \ge q r_N\}} \,. 
\end{equation}
With $\widehat{\delta}^{(i)}_N := e^{r_N} \delta^{(i)}_N$ we thus get 
from \eqref{dfn:Phi-N} and
$\log(1+y e^R) \le R + \log (1+y)$, for $R,y \ge 0$, that 
\[
\Phi_{N,\kappa} 
\le \frac{1}{N} \sum_{i=1}^N \log \big(1 + \widehat{\delta}^{(i)}_N 
e^{(1+q) {\widehat M}_\kappa^{(i)}} \big) + \frac{1}{q N} \sum_{i=1}^N 
\| {\bf g}^{(i)} \|_2^2 \one_{\{\| {\bf g}^{(i)} \|_2^2 \ge q r_N\}} 
:= \widehat \Phi_{N,\kappa} + \Psi_{N,\kappa} \,.
\]
Taking $\varphi := 1-q =\frac{2}{3} (\gamma-1)$ for 
$1 < \gamma < 5/2$ of \eqref{eq:J-assump-moments} 
and $r_N \to \infty$ slowly enough, we find that as $N \to \infty$,
\begin{equation}\label{dfn:hat-dN}
\widehat \delta_N := \frac{1}{N} \sum_{i=1} (\widehat \delta_N^{(i)})^\varphi 
\le (c_1 e^{r_N})^\varphi 
N^{-\gamma} \sum_{i,j=1}^N \E |J_{ij}|^3 \to 0 \,.
\end{equation}
Further, under $\widetilde \P^{\beta}_{N,\kappa}$ the 
variables $\{\widehat b_k^{(i)}\}$ are i.i.d.\ standard Gaussian, 
independent of ${\bf J}$. In particular, 
$(2 \widehat M_\kappa)^{1/2}$ is the Euclidean norm of a 
$\kappa$-dimensional, standard Gaussian random vector, which 
has the density $\hat{c}_\kappa r^{\kappa-1} e^{-r^2/2}$
at $r \in [0,\infty)$ for some $\hat{c}_\kappa < \infty$.
Thus, 
the elementary inequality $(1+u)^\varphi \leq 1 + u^{\varphi}$,
valid for $u\geq 0$ and $\varphi<1$, yields
the bound 
\[
\widetilde \E^{\beta}_{N,\kappa} \Big[\big( 1+ \widehat \delta e^{(1+q) \widehat M^{(i)}_\kappa}\big)^{\varphi} 
\Big] 
\le 
1 + (\widehat \delta)^{\varphi}\, \hat{c}_\kappa \int_0^\infty r^{\kappa-1} 
e^{-(q r)^2/2} dr 
= 1 + (\widehat \delta)^{\varphi} \,q^{-\kappa} \,.
\]
Combining this with Markov's inequality, yields for the i.i.d.\ 
$\widehat{M}_\kappa^{(i)}$, that 
\[
\widetilde \P^{\beta}_{N,\kappa} 
(\widehat \Phi_{N,\kappa} >  \eta) \le  
e^{-\varphi \eta N}
\prod_{i=1}^N \widetilde \E^{\beta}_{N,\kappa} \Big[\big( 1+ \widehat \delta_N^{(i)} e^{(1+q) \widehat M^{(i)}_\kappa}\big)^{\varphi} \Big] 
\le e^{-N (\varphi \eta - q^{-\kappa} \widehat{\delta}_N)} \,.
\]
In view of \eqref{dfn:hat-dN}, we thus deduce that 
\begin{equation}\label{eq:hat-PhiN-summable}
\sum_{N=1}^\infty  \widetilde \P^{\beta}_{N,\kappa} 
(\widehat \Phi_{N,\kappa} >  \eta) < \infty 
\end{equation}
and complete the proof of the lemma 
upon 
checking that for any $\eta>0$,
\begin{equation}\label{eq:PsiN-summable}
\sum_{N=1}^\infty  \widetilde \P^{\beta,a_2}_{N,\kappa} 
(\Psi_{N,\kappa} >  \eta) < \infty \,.
\end{equation}
To this end, first note that if $\|{\bf A}\|_{2 \to 2} \le a_2$, then necessarily
\begin{equation}\label{eq:RN-L1-bd}
\sum_{i=1}^N \|{\bf g}^{(i)}\|_2^2 = 
\frac{T}{\kappa}
\sum_{k=1}^{\kappa}  \|{\bf A} \widetilde X_{\ft_{k-1}}\|_2^2 \le T a_2^2 \fs^2 N \,.
\end{equation}
Thus, the random set 
$\cS_\star := \{ i \le N : \| {\bf g}^{(i)} \|_2^2 \ge q r_N \}$ has at most 
\[\ell_N := \lceil T a_2^2 \fs^2 N / (q r_N) \rceil = o(N)\] 
elements. By the union bound over the 
at most $\binom{N}{\ell_N} = \exp(o(N))$ ways to choose a non-random 
set $\cS \subseteq [N]$ of size $\ell_N$, it suffices for 
\eqref{eq:PsiN-summable} to show that 
\begin{equation}\label{eq:RS-tail}
\lim_{N \to \infty} \sup_{|\cS| = \ell_N} \frac{1}{N} \log  
\widetilde \P^{\beta}_{N,\kappa} (R_{\cS} >  q \eta N, \cA_{a_2}) < 0
\,, \quad {\rm where} \quad 
R_{\cS} := \sum_{i \in \cS} \| {\bf g}^{(i)} \|_2^2 \,.
\end{equation}
To this end, fixing $\cS \subset [N]$ of size $\ell_N$, consider the measure 
$\widetilde \P_{N,\kappa}^{\beta;\cS}$ where we set $\beta=0$ 
at all coordinates $i \in \cS$ of \eqref{eq:main_eq_frozen}, while not 
changing the value of $\beta$ when $i \notin \cS$. One then has
similarly to \eqref{eq:rdn-beta} the following Radon--Nykodim derivative,
expressed in terms of ${\bf b}^{(i)}$ of \eqref{dfn:si} and 
${\bf g}^{(i)}$ of \eqref{dfn:gt} by  
\begin{align}\label{eq:rdn-S}
\frac{{\rm d} \,\widetilde \P_{N,\kappa}^{\beta}}
{{\rm d}\, \widetilde\P_{N,\kappa}^{\beta;\cS}} 
& = \exp \Big( \sum_{i \in \cS}  
\langle {\bf b}^{(i)}, {\bf g}^{(i)} \rangle - \frac{1}{2} R_{\cS} \Big) \,. 
\end{align}
The \abbr{rhs} of \eqref{eq:rdn-S} is bounded for any $\theta>0$ (using the trivial bound $xy\leq (x^2+y^2)/2$ for $x=(1+\theta)^{-1/2}{\bf b}^{(i)}$ and $y=(1+\theta)^{1/2}{\bf g}^{(i)}$) by
\[
\exp\Big( \frac{M_{\cS}}{1+\theta} + \frac{\theta}{2} R_{\cS} \Big) \,,
\quad {\rm where } \quad 
M_{\cS} := \frac{1}{2} \sum_{i \in \cS} \| {\bf b}^{(i)}\|^2\,.
\]
In addition, $\max_{i} \{ N^{-1/2} \sum_{j=1}^N |A_{ij}| \}
\le \|{\bf A}\|_{2 \to 2}$ for any $N$-dimensional matrix,
yielding for $a:=a_2/\beta$, that   
\[
\cA_{a_2} \subseteq \cA^{\cS}_a := 
\bigcap_{i \in \cS} \{ \sum_{j=1}^N |J_{ij}| \le a N \} \,.
\]
Combining the preceding bounds, we arrive at
\begin{align*}
\widetilde \P^{\beta}_{N,\kappa} (R_\cS >  q \eta N, \cA_{a_2})
&\le \widetilde \E^{\beta;\cS}_{N,\kappa} \Big[
\exp\big(\frac{M_{\cS}}{1+\theta} + \frac{\theta}{2} R_{\cS}\big) 
\one_{\{R_\cS > q \eta N\}} \one_{\cA^\cS_{a}} \Big] \nonumber \\
&\le e^{-\theta q \eta N/2} \,
 \widetilde \E^{\beta;\cS}_{N,\kappa} \Big[
\exp\big(\frac{M_{\cS}}{1+\theta} + \theta R_{\cS}\big) \one_{\cA^\cS_{a}} \Big] \,.
\end{align*}
Under $\widetilde \P^{\beta;\cS}_{N,\kappa}$ the
variables $\{b_k^{(i)}, i \in \cS, k \le \kappa \}$ 
of \eqref{dfn:si} are i.i.d.\ standard Gaussian. With $M_\cS$ 
being the sum of half the squares of these variables, we clearly 
have that for some $f_0(\cdot)$ finite, any $\theta>0$ and all $\kappa,\cS$,
\[
\widetilde \E^{\beta;\cS}_{N,\kappa} \big[ e^{M_{\cS}/(1+\theta)} \big] 
= f_0(\theta)^{\kappa |\cS|} \,.
\]
Denoting by $\cF_N :=\sigma( b_k^{(i)}, x_{kj}, k \le \kappa, i,j \le N)$ 
the $\sigma$-algebra generated by $b_k^{(i)}$ of \eqref{dfn:si} 
and $\{x_{kj} \}$ of \eqref{dfn:gt}, it thus 
suffices for \eqref{eq:RS-tail} to show the existence of 
\emph{non-random} $\theta>0$, $N_1$ and $f_1=f_1(\theta)<\infty$, 
such that for all $N \ge N_1$ and $\cS \subseteq [N]$,
\begin{equation}\label{eq:unif-R-mgf}
\widetilde \E^{\beta;\cS}_{N,\kappa} \Big[
e^{\theta R_{\cS}} \one_{\cA^\cS_a} \mid \cF_N \Big]  \le f_1(\theta)^{|\cS|} \,.
\end{equation}
To this end, recall that under $\widetilde \P^{\beta;\cS}_{N,\kappa}$ the 
vectors $\{{\bf J}_i := (J_{ij}) \in \R^N,\, i\in \cS\}$ of mutually 
independent entries are \emph{independent} of $\cF_N$.
Further, in view of \eqref{dfn:gt}, 
\[
R_\cS = \sum_{i \in \cS} 
\sum_{k=1}^{\kappa} \langle {\bf x}_k, {\bf J}_i \rangle^2 \,,
\] 
where ${\bf x}_k =(x_{kj}) \in \R^N$ is such that  
$\|{\bf x}_k \|_\infty \le \sqrt{T/(N \kappa)} \beta \fs$. We thus have that 
\begin{align*}
\widetilde \E^{\beta;\cS}_{N,\kappa} \Big[
e^{\theta R_{\cS}} \one_{\cA^\cS_a} \, | \, \cF_N \Big] &\le f_{N,\kappa}
(\theta T \beta^2 \fs^2)^{|\cS|} \,,
\\
f_{N,\kappa} (\alpha) &:= \max_{i \le N}
\sup_{\{{\bf u}_k \in \R^N :   \|{\bf u}_k \|_\infty \le N^{-1/2} \}} \, \E \big[
\exp \big( \frac{\alpha}{\kappa} 
\sum_{k=1}^{\kappa} \langle {\bf u}_k, {\bf J}_i \rangle^2 \big) \one_{\cA^{\{i\}}_a} 
\big]  \,.
\end{align*}
By Jensen's inequality $f_{N,\kappa}(\alpha) \le f_{N,1}(\alpha)$.
Further, thanks to our assumption 
\eqref{eq:J-assump-mgf}, the vectors ${\bf J_i}$ of 
independent coordinates satisfy the condition \eqref{eq:expon-tail} 
for some $\epsilon>0$, $v<\infty$ which are independent of $i$ and $N$. 
Hence, taking $\theta>0$ so $\alpha=\theta T \beta^2 \fs^2$ be as 
in Lemma \ref{lem:mgf-quadratic}, results for $N \ge N_1$ with 
$f_{N,\kappa}(\alpha) \le f_\star(\alpha v)$ finite, thus
establishing \eqref{eq:unif-R-mgf} and thereby concluding the proof of the lemma.
\end{proof}

Setting $\widehat W_T^\fs$ as the metric space $C([0,T]\to[-\fs,\fs])$ equipped
with the distance
\[
d_{\infty}(x,y) = \sup_{t\in[0,T]}|x(t)-y(t)| \,,
\]
we next effectively prove an exponential tightness of $\widetilde \Pi^{\beta}_{N,\kappa}$
in the corresponding weak topology (using an entropy bound, this has been 
proved in \cite{BG98} for Gaussian disorder).
\begin{lemma}\label{lem:tight}
Fixing $T,\beta,\kappa,a_2,\alpha$, there exists  
$\cK_\alpha \subset \cM_1(\widehat{W}_T^\fs)$ compact, with
\[
\limsup_{N \to \infty} \frac{1}{N} \log 
 \widetilde \Pi^{\beta,a_2}_{N,\kappa} (\cK_\alpha^c) 
< -\alpha \,.
\]
\end{lemma}
\begin{proof} It follows from \eqref{eq:rdn-beta} that for $M_\kappa^{(i)}:=\frac{1}{2} \|{\bf b}^{(i)}\|^2_2$ and ${\bf b}^{(i)}$ of \eqref{dfn:si},
\[
\frac{{\rm d} \,\widetilde \P_{N,\kappa}^\beta}{{\rm d}\, \P_{N}^0} 
\le \exp\Big(\sum_{i=1}^N M_\kappa^{(i)}\Big)\,.
\]
By \eqref{eq:RN-L1-bd} and the \abbr{lhs} of \eqref{eq:M-to-whM}, having 
$\|{\bf A}\|_{2 \to 2} \le a_2$ yields that
\[
\sum_{i=1}^N M_\kappa^{(i)} \le  
2 \sum_{i=1}^N {\widehat M}_\kappa^{(i)} + T a_2^2 \fs^2 N  \,.
\]
Further, under $\widetilde \P^{\beta}_{N,\kappa}$ the variables $\{\widehat b_k^{(i)}\}$
are i.i.d.\ standard Gaussian, independent of ${\bf A}$, hence
for any $\cA \subset \cM_1(\widehat{W}^\fs_T)$ 
and $r' = (r - T a_2^2 \fs^2)/2 \ge 1$,
\begin{equation}\label{eq:coarse-bd-rnd}
\widetilde \Pi^{\beta,a_2}_{N,\kappa} (
\cA)
\le e^{\kappa r N} \, \Pi^{0}_{N} (
\cA)
+ \widetilde \Pi^{\beta}_{N,\kappa} (\sum_{i=1}^N 
\widehat{M}_\kappa^{(i)} \ge \kappa r' N) \,.
\end{equation}
The \abbr{cgf} 
$\log \widetilde \E_{N,\kappa}^{\beta} [e^{\theta {\widehat M}^{(1)}_\kappa}] 
= \kappa \Lambda(\theta)$ which is independent of $N$ and ${\bf A}$ 
(hence also on $\beta$), is finite at $\theta<1$ and has
$\kappa \Lambda'(0) = \widetilde \E_{N,\kappa}^{\beta}
\widehat {M}_\kappa = \kappa/2$. Thus,   
$\theta \ge \Lambda(\theta)$ for small enough $\theta>0$, so applying 
Markov's inequality, we get for such $\theta>0$ and i.i.d.\  ${\widehat M}_\kappa^{(i)}$, 
that for some $\widehat{r}=\widehat{r}(\alpha,\kappa)$ finite,
\begin{equation}\label{eq:exp-whM}
\widetilde \Pi^{\beta}_{N,\kappa} 
(\sum_{i=1}^N {\widehat M}_\kappa^{(i)} \ge \kappa \widehat{r} N) \le 
\exp( - N \kappa [\theta \widehat{r} - \Lambda(\theta)] )  < \exp( - \alpha N)\,.
\end{equation}
Thus, thanks to \eqref{eq:coarse-bd-rnd} and \eqref{eq:exp-whM},
it suffices to verify that $\Pi_N^{0}$
are exponentially tight in $\cM_1(\widehat{W}_T^\fs)$. To this end,
recall that this is the law of the empirical measure $\mu_N$ of 
independent $(X_\cdot^{(i)})$ (namely, the solutions of the \abbr{sds}~\eqref{eq:main_eq} which are uncoupled at~$\beta=0$). 
These i.i.d.\ variables take value in a Polish space~$\widehat{W}_T^\fs$, whence
$\Pi_N^{0}$ is exponentially tight in the induced weak topology 
(see~\cite[Lemma 6.2.6]{DZ98}).
\end{proof}

We have the following upon combining \cite{BG95} and Lemma \ref{lem:tight}.
\begin{lemma}\label{lem:weak-ld-gauss}
For every $\beta,\kappa,T,a_2$ there exists a good rate function $I_\kappa$ 
on $\cM_1(\widehat W_T^\fs)$
such that, for any closed $\cF \subset \cM_1(\widehat W_T^\fs)$, 
\[ 
\limsup_{N\to\infty}\frac1N \log\widehat\Pi_{N,\kappa}^{\beta,a_2} 
(\cF) \leq -I_\kappa(\cF)\,.
\]
In addition, 
$I_{\kappa}(\cF)\to I(\cF)$ as $\kappa\to\infty$ and $I(\cdot)$ is a good rate function 
whose unique minimizer is $\mu_\star$.
\end{lemma}
\begin{proof} The large deviations 
upper bound with a good rate function $I_\kappa(\cdot)$ is established  
in \cite[Theorem~3.1(1)--(2)]{BG95} for $\cF$ compact and 
$\widehat\Pi_{N,\kappa}^{\beta,a_2}$. The exponential 
tightness from Lemma \ref{lem:tight} applies in particular for the Gaussian 
disorder $\widehat {\bf J}$, thereby extending the validity of such upper bound 
to all closed sets $\cF$. Finally, \cite[Prop.~4.3(1)]{BG95} shows
the convergence of $I_\kappa(\cdot)$ to some 
$I(\cdot)$ whose global minimizer $\mu_\star$ is unique.
\end{proof}

\begin{proof}[\textbf{\emph{Proof of Proposition~\ref{thm-disc}}}] We actually prove a slightly stronger 
statement, where $\B(\mu_\star,\delta)$ denotes instead the ball of radius $\delta>0$
and center $\mu_\star$ in $\cM_1(\widehat{W}^\fs_T)$. Fixing $T,\beta,\delta,a_2$,
we have by Lemma \ref{lem:compare} and the union bound, that for any $\kappa$,
$N \ge N_0(\kappa)$ and all $\eta>0$,
\[
\widetilde\Pi_{N,\kappa}^{\beta,a_2} ( \B (\mu_\star,\delta)^c) \le 
\widetilde \P_{N,\kappa}^{\beta,a_2} (\Phi_{N,\kappa} > 2 \eta) + 
e^{2 \eta N} \, \widehat \Pi_{N,\kappa}^{\beta,a_2} (\B (\mu_\star,\delta)^c) \,.
\]
In view of Lemma \ref{lem:Phi-summable} it thus suffices to 
show that for any $\delta >0$ and all $\kappa \ge \kappa_0(\delta)$ 
\begin{equation}\label{eq:ldp-ubd1}
\limsup_{N \to \infty} \frac{1}{N} \log 
\widehat \Pi_{N,\kappa}^{\beta,a_2} (\B (\mu_\star,\delta)^c)
< 0 \,.
\end{equation}
Recall from Lemma \ref{lem:weak-ld-gauss}, that $I(\cF_\delta)>0$
for the closed set $\cF_\delta = \B(\mu_\star,\delta)^c$ and therefore 
$I_\kappa(\cF_\delta) > 0$ for all $\kappa \ge \kappa_0(\delta)$. We thus get \eqref{eq:ldp-ubd1} and thereby complete the proof of 
the theorem, upon considering the \abbr{ldp} upper bound of 
Lemma \ref{lem:weak-ld-gauss} for this $\cF_\delta$.
\end{proof}

\section{Proof of Proposition \ref{thm-coupling}}\label{sec:couple}
For $\beta>0$, we couple $X_t$ and $\widetilde{X}_t$ using the same Brownian motion: 
writing \[ \cE_t := \widetilde{X}_t-X_t \]
we see that
\begin{align*}
\frac{d}{dt}\cE_t &= \nabla U(X_t)-\nabla U(\widetilde{X}_t) + \frac\beta{\sqrt{N}}{\bf J} (\widetilde{X}_{\ft_{\lfloor t \kappa / T\rfloor}}-X_t)\,,\\
\cE_0 &= 0\,.
\end{align*}
Let
\[ R_t = \| \cE_t\|_2\quad\mbox{and}\quad {\bf A} = \frac{\beta}{\sqrt{N}}{\bf J}\,.\]
Then
\begin{align*}  R_t \frac{d}{dt} R_t &=
\Big< \cE_t,\frac{d}{dt} \cE_t\Big>
\\ &=
 \left< \cE_t, {\bf A} \cE_t\right> + \left< \cE_t,\nabla U(X_t) - \nabla U(\widetilde{X}_t)\right> + \left<\cE_t, {\bf A}(\widetilde{X}_{\ft_{\lfloor t \kappa / T\rfloor}} - \widetilde{X}_t)\right>\,,
\end{align*}
which, by the mean value theorem, is at most
\[ \|{\bf A}\|_{2\to 2} R_t^2 + R_t^2 c'' + \|{\bf A}\|_{2\to 2} R_t L_t\,,
\]
where, for fixed $\epsilon>0$ and $\rho>0$, we define
\[ L_t = \| \widetilde{X}_{\ft_{\lfloor t \kappa / T\rfloor}} - \widetilde{X}_t\|_2
\qquad\mbox{ and }\qquad c'' = \sup_{|x|\leq \fs}(-U_1''(x))\,.
\]
Restricted to the event $\cA_{a_2}$, we have that 
\begin{equation}\label{eq:R't-up-to-stopping-times}
\frac{d}{dt}R_t \leq \left(a_2 + c''\right)R_t + 3 a_2 \rho \sqrt{N}\,,
\end{equation}
up to the stopping time 
\[
\widetilde\tau_\rho := \inf\{t\geq 0\,:\; L_t\geq  3\rho \sqrt{N}\}\,.
\]
Solving the \abbr{ode} that corresponds to equality in
\eqref{eq:R't-up-to-stopping-times}, starting at $R_0=0$, 
results with  
\[ 
R_t \leq \frac{3a_2}{a_2+c''}(e^{(a_2+c'')t}-1) \rho \sqrt{N} \le \delta \sqrt{N} 
\]
up to $\widetilde \tau_\rho$, provided that
\[ \rho \leq \delta \bigg(\frac{a_2+c''}{3a_2}\bigg) e^{-(a_2+c'')T}\,.\]
In particular, for such $\rho=\rho(\delta,a_2,T)>0$ it then follows that 
\begin{align*} \Q^{\beta,a_2}_N \Big(\sup_{t\in[0,T]} R_t \geq \delta\sqrt{N}\Big) &\leq  \P^{\beta,a_2}_N(\widetilde{\tau}_\rho \leq T \,, D_\epsilon^c) + 
\P^{\beta,a_2}_N (D_\epsilon) \,,
\end{align*}
for any $\epsilon>0$, where 
\begin{align*}
      \widetilde\sigma_{\epsilon}^{(i)} &= \inf\{ t\geq 0 \,:\; |\widetilde{X}_t^{(i)}| > \fs-\epsilon\}\,,\qquad     D_{\epsilon} = \Big\{ \sum_{i=1}^N \one_{\{\widetilde\sigma_\epsilon^{(i)} < T\}} > \frac{\rho^2}{\fs^2} N\Big\} \,.
\end{align*}
Recall from \cite[Thm.~4.1(a)]{BG95} that $\mu_\star$ is absolutely continuous 
w.r.t.\ the law $\P_1^0$ of the solution $X^{(1)}_t$, $t \in [0,T]$, of a single \abbr{sde} \eqref{eq:main_eq} at $\beta=0$. Further,
$\P_1^0(\{x : \|x\|_\infty=\fs\})=0$ thanks to \eqref{eq:U-assump}, hence 
for any small $\epsilon>0$ and the corresponding closed subset 
\begin{equation}\label{eq:Feps}
\mu_\star \notin \cF_\epsilon := 
\Big\{ \mu : \mu( \|x\|_\infty \ge \fs - \epsilon ) \ge \frac{\rho^2}{\fs^2} \Big\} \,.
\end{equation}
Fixing such $\epsilon
>0$, we proceed to bound $L_t$. To this end,
recall that for any $t\in[\ft_k,\ft_{k+1}]$,
\[ \widetilde{X}_t - \widetilde{X}_{\ft_k} = -\int_{\ft_k}^t \nabla U(\widetilde{X}_\xi)d\xi + B_t - B_{\ft_k} + (t-\ft_k) {\bf A} \widetilde{X}_{\ft_k}\,.\]
Hence, with $t-\ft_k \leq T/\kappa$, setting  
$c'_\epsilon := \sup_{|x|\leq \fs-\epsilon}|U_1'(x)|$ we get that 
\[
\sum_{i=1}^N \one_{\{\widetilde{\sigma}_\epsilon^{(i)} \ge  T\}}\big|\widetilde{X}^{(i)}_t-\widetilde{X}^{(i)}_{\ft_k}\big|^2
\leq 
\bigg[c'_\epsilon \frac{T}\kappa \sqrt{N} + \|B_t-B_{\ft_k}\|_2 + \frac{T}\kappa\|{\bf A}\|_{2\to 2}\, \|\widetilde{X}_{\ft_k}\|_2\bigg]^2\,.
\]
At the same time, on the event $D_\epsilon^c$ we have that 
\[
\sum_{i=1}^N \one_{\{\widetilde{\sigma}_\epsilon^{(i)} < T\}}\big|\widetilde{X}^{(i)}_t-\widetilde{X}^{(i)}_{\ft_k}\big|^2
\leq (2\rho)^2 N\,.
\]
Consequently, (using that $\|\widetilde{X}_{\ft_k}\|_2 \leq \fs\sqrt N$ and the restriction to the event $\cA_{a_2}$) we deduce that as soon as 
\[ \kappa \geq \kappa_1 := \left\lceil (c'_\epsilon + a_2 \fs) T/\rho\right\rceil \]
we have by the independence of the Brownian increments (and a union bound),
\[ \P_{N}^{\beta,a_2} (\widetilde{\tau}_\rho \leq T \,,\, D_\epsilon^c) 
\leq \kappa \,\P\bigg( \sup_{t\in[0,\frac{T}{\kappa}]} \{ \|B_t\|^2_2 \} \geq \rho^2  N\bigg)\,.
\]
To bound the latter, note that 
\[ u(t,x) = \exp\bigg[ \frac{x^2}{1+2t}-\frac12\log(1+2t)\bigg]\,,\]
is a positive, smooth solution of the heat equation $u_t+\frac12u_{xx}=0$.
Hence, by Ito's formula the integrable $M_t^{(i)} := u(t,B_t^{(i)})$ are i.i.d.\ 
positive martingales, starting at $M_0^{(i)} = 1$. Next, increasing 
$\kappa_1$ as needed in order to have 
\[ \eta := \frac{\rho^2}{1+2T/\kappa_1} - \frac12\log(1+2T/\kappa_1) > 0 \,,\]
and applying Doob's maximal inequality for the positive martingale 
$\overline{M}_t = \prod_{i=1}^N M_t^{(i)}$,
we deduce that for any $\kappa\geq \kappa_1$,
\[\P\bigg( \sup_{t\in[0,\frac{T}{\kappa}]} \{ \|B_t\|^2_2 \} \geq \rho^2  N\bigg) \leq \P\bigg(\sup_{t\in[0, \frac{T}\kappa]} \{ \overline M_t \} \geq e^{\eta N}\bigg) \leq 
e^{-\eta N} \,.
\]
Turning now to show that $\P_{N}^{\beta,a_2}(D_\epsilon)$ is summable in $N$ for all
$\kappa \ge \kappa_1(\epsilon,\rho)$, note that 
$D_\epsilon \subseteq \{ \widetilde \mu_{N,\kappa} \in \cF_\epsilon \}$
for $\cF_\epsilon$ of \eqref{eq:Feps}. Thus, proceeding as in the proof of 
Proposition \ref{thm-disc}, 
we have by Lemma \ref{lem:compare} and the union bound, that for any $\kappa$,
$N \ge N_0(\kappa)$ and all $\eta>0$,
\[
 \P_{N}^{\beta,a_2}( D_\epsilon) 
\le \widetilde \Pi^{\beta,a_2}_{N,\kappa} (\cF_\epsilon) \le 
 \widetilde\P^{\beta,a_2}_{N,\kappa} (\Phi_{N,\kappa} > 2 \eta) + 
e^{2 \eta N} \, \widehat \Pi_{N,\kappa}^{\beta,a_2}( \cF_\epsilon) \,.
\]
Thanks to Lemma \ref{lem:Phi-summable}, it thus suffices to 
establish that for all $\kappa \ge \kappa_1$, 
\begin{equation}\label{eq:ldp-ubd}
\limsup_{N \to \infty} \frac{1}{N} \log 
\widehat \Pi_{N,\kappa}^{\beta,a_2}( \cF_\epsilon)
< 0 \,,
\end{equation}
which as we have seen before, follows from Lemma \ref{lem:weak-ld-gauss} 
since $\mu_\star \notin \cF_\epsilon$ (hence $I(\cF_\epsilon)>0$).

\subsection*{Acknowledgment} 
We thank G.\ Ben Arous and A.\ Guionnet for a valuable feedback on our preliminary draft and for 
 pointing our attention to the references~\cite{BG98,FMT19,Grunwald96,Rieger1989}.
A.D.~was supported in part by NSF grant DMS-1613091 and E.L.~was supported in part by 
NSF grant DMS-1812095. 
This research was further supported in part by BSF grant 2018088.

\bibliographystyle{abbrv}
\bibliography{langevin_ref}

\end{document}